\documentclass[12pt,reqno]{amsart}
\usepackage{amssymb}
\usepackage{amsfonts}
\usepackage{graphicx}
\usepackage{color}
\usepackage{algpseudocode}
\usepackage{hyperref}
\usepackage{tikz}
\usetikzlibrary{shapes,arrows}

\tikzstyle{process} = [rectangle, minimum width=9cm, minimum
height=1cm, text width=9cm, text centered, draw=black]
\tikzstyle{arrow} = [thick,->,>=stealth]

\newtheorem{theorem}{Theorem}[section]

\newtheorem{lemma}[theorem]{Lemma}

\def\Z{\mathbb{Z}}

\newcommand{\Aut}{\mathrm{Aut}}

\begin{document}

\title{Quasi-symmetric designs on $56$ points}

\author{Vedran Kr\v{c}adinac}

\email{vedran.krcadinac@math.hr}

\author{Renata Vlahovi\'{c} Kruc}

\email{renata.vlahovic@math.hr}

\address{Department of Mathematics, Faculty of Science, University of Zagreb,
Bijeni\v{c}ka~30, HR-10000 Zagreb, Croatia}

\thanks{This  work  has  been  fully  supported by
the Croatian Science Foundation under the project $6732$.}

\subjclass[2000]{05B05}

\keywords{quasi-symmetric design; automorphism group; construction
method}

\date{February 28, 2020}

\begin{abstract}
Computational techniques for the construction of quasi-symmetric
block designs are explored and applied to the case with $56$ points.
One new $(56,16,18)$ and many new $(56,16,6)$ designs are
discovered, and non-existence of $(56,12,9)$ and $(56,20,19)$
designs with certain automorphism groups is proved. The number of
known symmetric $(78,22,6)$ designs is also significantly increased.
\end{abstract}

\maketitle

\section{Introduction}

A \emph{block design} with parameters $(v,k,\lambda)$ consists of a
set of $v$ \emph{points} and a family of $k$-element subsets called
\emph{blocks} such that every pair of points is contained in
$\lambda$ blocks. The number of blocks through a single point
$r=\lambda\cdot \frac{v-1}{k-1}$ and the total number of blocks
$b=\lambda\cdot \frac{v(v-1)}{k(k-1)}$ are also determined by the
parameters. The \emph{degree} of a design is the number of
cardinalities $|B_1\cap B_2|$ occurring as intersections of two
blocks $B_1$ and $B_2$. \emph{Symmetric designs} are designs with
$v=b$ or, equivalently, designs of degree~$1$. A design is called
\emph{quasi-symmetric} if it is of degree~$2$, i.e.\ if any pair of
blocks intersects in $x$ or in $y$ points, for some integers $x<y$.
We refer to~\cite{BJL99} for results about block designs and
generalisations such as $t$-designs and partially balanced designs,
and to~\cite{MS07, SS91} for results about quasi-symmetric designs
(QSDs).

Neumaier~\cite{AN82} published a table of admissible parameters of
QSDs with $v\le 40$. It contains $15$ parameter sets for which the
existence of QSDs was unknown at the time. Although all but $3$ of
these parameter sets have been eliminated in the meantime, the
recent construction of $(56,16,18)$ QSDs~\cite{KV16} gives hope that
there may be other parameters for which QSDs exist, but have not
been discovered yet. In this paper we focus on quasi-symmetric
designs on $56$ points and computational techniques for their
construction. We explore known methods relying on assumed
automorphism groups and enhance them sufficiently to be able to
thoroughly examine the case $v=56$. One new $(56,16,18)$ and many
new $(56,16,6)$ QSDs are constructed, and non-existence of
$(56,12,9)$ and $(56,20,19)$ QSDs with certain automorphism groups
is proved.

The layout of our paper is as follows. In Section~\ref{sec2}, we
survey known results about QSDs on $56$ points and eliminate two
parameter sets by a theorem of Calderbank~\cite{AC88}.
Section~\ref{sec3} is devoted to $(56,16,18)$ QSDs. An algorithm for
generating good orbits of $k$-element subsets is developed. Together
with an approach based on clique search, it is used to classify
$(56,16,18)$ QSDs with a permutation group $G_{48}$ of order~$48$,
resulting in a new design. In Section~\ref{sec4}, the group $G_{48}$
is used to construct $876$ new $(56,16,6)$ QSDs. A complete
classification of these designs with the Frobenius group $Frob_{21}$
of order $21$ is carried out by computations based on orbit
matrices. This technique is also used to construct $303$ new
$(56,16,6)$ QSDs with automorphism groups isomorphic to the
alternating group $A_4$. More examples are constructed from the
binary linear codes associated with the constructed designs. The new
$(56,16,6)$ QSDs significantly increase the number of known
symmetric $(78,22,6)$ designs, in which they can be embedded as
residual designs. In Section~\ref{sec5}, the developed construction
techniques are applied to $(56,12,9)$ and $(56,20,19)$ QSDs with
automorphism groups from the previous sections. It is shown that
$G_{48}$ cannot be an automorphism group of these QSDs, and
$(56,12,9)$ QSDs with $Frob_{21}$ and some subgroups of $G_{48}$ are
also eliminated. The results of our computations are summarised in
the final Section~\ref{sec6}.

The paper relies heavily on computer calculations. We use
GAP~\cite{GAP4} for group calculations, and nauty~\cite{MP14} to
check isomorphism and compute full automorphism groups of designs.
Cliquer~\cite{NO03, PO02} is used to search for cliques in weighted
graphs, and Magma~\cite{BCP97} for calculations with codes
associated with the designs. The most time-critical calculations are
performed with our own programs written in the C language.

\section{The known QSDs on $56$ points}\label{sec2}

M.~S.~Shrikhande's survey in the Handbook of Combinatorial
Designs~\cite{MS07} includes a table of admissible parameters of
quasi-symmetric designs with $v\le 70$ (Table 48.25).
We reproduce the six rows relating to designs on $56$ points in
Table~\ref{table1}, with updated information on the number of
designs 
in column NQSD.

\begin{table}[h!]
\begin{tabular}{cccccccccc}
\hline
No. & $v$ & $k$ & $\lambda$ & $r$ & $b$ & $x$ & $y$ & NQSD & Ref.\\
\hline
47 & 56 & 16 & 18 & 66 & 231 & 4 & 8 & $\ge 3$ & \cite{KV16}\\
48 & 56 & 15 & 42 & 165 & 616 & 3 & 6 & 0 & \cite{AC88}\\
49 & 56 & 12 & 9 & 45 & 210 & 0 & 3 & ? & \\
50 & 56 & 21 & 24 & 66 & 176 & 6 & 9 & 0 & \cite{AC88}\\
51 & 56 & 20 & 19 & 55 & 154 & 5 & 8 & ? & \\
52 & 56 & 16 & 6 & 22 & 77 & 4 & 6 & $\ge 2$ & \cite{VT87, MT04}\\
\hline
\end{tabular}
\vskip 3mm \caption{The known QSDs with $v=56$.}\label{table1}
\end{table}

The first QSD on $56$ points was constructed by Tonchev~\cite{VT87}
by embedding the $3$-$(22,6,1)$ design in a symmetric $(78,22,6)$
design. The corresponding residual design is quasi-symmetric with
parameters $(56,16,6)$, $x=4$, $y=6$, appearing in row no.\ 52. 
Another QSD with these parameters was constructed in~\cite{MT04}
from words of weight $16$ in the binary linear code associated with
the first design.

Three QSDs with parameters $(56,16,18)$, $x=4$, $y=8$ from row 47
were constructed in~\cite{KV16} by assuming a suitable automorphism
group. The question marks in rows 48 and 50 of \cite[Table
48.25]{MS07} can be eliminated by Calderbank's Theorem~2
from~\cite{AC88}. A version of the theorem specialised to
quasi-symmetric designs is reproduced here for the convenience of
the reader.

\begin{theorem}[Calderbank~\cite{AC88}]\label{calderbankthm}
Let $p$ be an odd prime and let $D$ be a $(v,k,\lambda)$ QSD with
block intersection numbers $x\equiv y \equiv s \pmod{p}$. Then
either
\begin{enumerate}
\item[$(1)$\,] $r\equiv \lambda \pmod{p^2}$,
\item[$(2)$\,] $v\equiv 0 \pmod{2}$, $v\equiv k\equiv s \equiv 0 \pmod{p}$,
$(-1)^{v/2}$ is a square in $GF(p)$,
\item[$(3)$\,] $v\equiv 1 \pmod{2}$, $v\equiv k\equiv s \not\equiv 0 \pmod{p}$,
$(-1)^{(v-1)/2}\, s$ is a square in $GF(p)$,
\item[$(4)$\,] $r\equiv \lambda \equiv 0 \pmod{p}$ and either
\begin{enumerate}
\item[(a)\,] $v\equiv 0 \pmod{2}$, $v\equiv k\equiv s \not\equiv 0 \pmod{p}$,
\item[(b)\,] $v\equiv 0 \pmod{2}$, $k\equiv s \not\equiv 0
\pmod{p}$, $v/s$ is a nonsquare in $GF(p)$,
\item[(c)\,] $v\equiv 1 \pmod{2p}$, $r\equiv 0 \pmod{p^2}$, $k\equiv
s \not\equiv 0 \pmod{p}$,
\item[(d)\,] $v\equiv p \pmod{2p}$, $k\equiv s \equiv 0 \pmod{p}$,
\item[(e)\,] $v\equiv 1 \pmod{2}$, $k\equiv s \equiv 0 \pmod{p}$,
$v$ is a nonsquare in $GF(p)$,
\item[(f)\,] $v\equiv 1 \pmod{2}$, $k\equiv s \equiv 0 \pmod{p}$,
$v$ and $(-1)^{(v-1)/2}$ are squares in $GF(p)$.
\end{enumerate}
\end{enumerate}
\end{theorem}

For $p=3$, the parameters $(56,15,42)$, $x=3$, $y=6$ and
$(56,21,24)$, $x=6$, $y=9$ do not satisfy any of the conditions of
Theorem~\ref{calderbankthm}, so these designs don't exist. This
seems to have been overlooked in Table~I of Calderbank's
paper~\cite{AC88} and the omission was copied in later editions of
the table, including~\cite[Table 48.25]{MS07}. Designs with
parameters from rows 49 and 51 cannot be eliminated by this theorem.
We shall attempt to construct them by computational techniques
relying on assumed automorphism groups and to increase the number of
known designs in rows 47 and 52.

Tables of admissible parameters of QSDs also appear in~\cite{PSN15},
organised by the associated strongly regular graphs (SRGs).
Parameters from rows 47 and 52 appear in \cite[Table 1]{PSN15},
where the SRGs are known to exist. However, parameters from rows 48
to 51 are missing from \cite[Table 3]{PSN15}, where existence of the
SRGs is unknown.

\section{A new $(56,16,18)$ QSD}\label{sec3}

Let $G$ be a permutation group on a $v$-element set, say
$V=\{1,\ldots,v\}$. Finding $(v,k,\lambda)$ designs with $V$ as the
set of points and $G$ as an automorphism group is done in two steps:
\begin{enumerate}
\item[1.] generate the orbits of $G$ on $k$-element subsets of
$V$,
\item[2.] select orbits comprising blocks of the design.
\end{enumerate}
For quasi-symmetric designs, only \emph{good} orbits need to be
considered, i.e.\ orbits containing $k$-element sets intersecting in
$x$ or $y$ points.

In~\cite{KV16}, a group $H$ isomorphic to
$(\Z_2)^4 \rtimes A_5$ was used to find three $(56,16,18)$
designs with intersection numbers $x=4$, $y=8$. Here and in the sequel,
$N\rtimes M$ denotes a semidirect product of groups $N$ and $M$, where
$N$ is normal in the product. The computation in~\cite{KV16} was
fairly small because the order of the group $|H|=960$ exceeds the number
of blocks of the design, $b=231$. Orbits of size greater than $b$ can
be omitted, and the ``short orbits'' can be generated efficiently by
an algorithm based on stabilisers, as described in~\cite{KV16}.

Let us now consider the subgroup $G_{48}=\langle
\alpha\beta\alpha^{-1}, \alpha^{-1}\beta\alpha^2,
\beta\alpha\beta\alpha^{-1}\beta\rangle$ $\cong (\Z_2)^4 \rtimes \Z_3$,
where $\alpha$ and $\beta$ are the generators of $H$ given
in~\cite{KV16}. Now we need to find orbits of all sizes, up to
$|G_{48}|=48$. We use an orderly algorithm of Read-Farad\v{z}ev type
\cite{IF78, RR78} (see also \cite{BM98}). Subsets of~$V$ are
compared lexicographically. Suppose $U=\{u_1,\ldots,u_k\}$ and
$W=\{w_1,\ldots,w_k\}$, for some $u_1<\ldots<u_k$ and
$w_1<\ldots<w_k$. Then $U<W$ provided there is an index $i$ such
that $u_i<w_i$ and $u_j=w_j$, for all $j<i$. We denote by $m(U)$ the
minimal element of the orbit $\{gU\mid g\in G\}$ with respect to
this total order. The call \textsc{GoodOrbits}($\emptyset$) of the
following recursive algorithm will output the minimal representative
from each good orbit of $k$-element subsets of $V$. \vskip 2mm

\begin{algorithmic}[1]
\Procedure{GoodOrbits}{$U$: subset of $V$}
  \If {$|U|=k$}
    \If {$|U\cap gU|\in\{x,y,k\}$ for all $g\in G$}
      \State \textbf{print} U
    \EndIf
  \Else
    \For {$e=\max U+1,\ldots,v$}
      \If {$m(U\cup\{e\})=U\cup\{e\}$}
         \State \Call{GoodOrbits}{$U\cup\{e\}$}
      \EndIf
    \EndFor
  \EndIf
\EndProcedure
\end{algorithmic}
\vskip 2mm

The correctness of the algorithm is based on the following observation.
If $U'$ is minimal, i.e.\ $m(U')=U'$, and $U$ is obtained by
removing the largest element of $U'$, then $U$ is also minimal
($m(U)=U$). A program written in C based on this algorithm needs
about $6$ days of CPU time on a $2.66$ GHz processor to generate all
$G_{48}$-orbits of $16$-element subsets of $V=\{1,\ldots,56\}$. The
total number of orbits is $867\,693\,085\,859$ and there
are $301\,080$ good orbits (with intersection numbers $x=4$, $y=8$)
among them.

The second step of the computation is usually more difficult.
In~\cite{KV16}, we used the Kramer-Mesner approach based on solving
systems of linear equations over $\{0,1\}$, a known NP complete
problem. The equations correspond to the requirement that
$2$-element subsets of $V$ are covered exactly $\lambda$ times by
blocks of the design (balancedness). Together with information on
compatibility of the orbits, corresponding to the requirement that
blocks intersect in $x$ or $y$ points, we could handle problems of
about $1000$ orbits in~\cite{KV16}.

Now we have significantly more orbits and use a different approach
based on clique search, another NP complete problem. We define a
graph with the good orbits as vertices and edges between compatible
orbits. This is the \emph{compatibility graph} of the orbits.
To each vertex we assign a weight equal to the size of the orbit,
and use the program Cliquer~\cite{NO03, PO02} to find all cliques of
weight $b$ in this graph. The cliques correspond to families of $b$
subsets of size~$k$, intersecting in~$x$ or $y$ points. In the end
we check which of the families are balanced, i.e.\ designs, and
eliminate isomorphic copies. A similar approach was used
in~\cite{CRRT17} and~\cite{MT04}. The calculation is summarised by
the following flowchart.

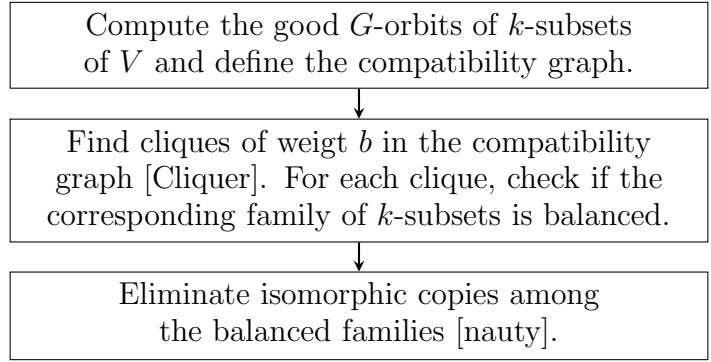
\begin{center}
\noindent \begin{tikzpicture}[node distance=1.8cm]
\node (step1) [process] {Compute the good $G$-orbits of $k$-subsets of
$V$ and define the compatibility graph.};
\node (step2) [process, below of=step1] {Find cliques of weigt $b$ in
the compatibility graph [Cliquer]. For each clique, check if the
corresponding family of $k$-subsets is balanced.};
\node (step3) [process, below of=step2] {Eliminate isomorphic copies
among the balanced families [nauty].};
\draw [arrow] (step1) -- (step2);
\draw [arrow] (step2) -- (step3);
\end{tikzpicture}
\end{center}

For the group $G_{48}$, the compatibility graph has $301\,080$
vertices and $21\,193\,946$ edges (density $4.676\cdot 10^{-4}$).
Cliquer needs about $2$ hours of CPU time to find all $1\,049\,792$
cliques of weight $231$. Among them there are $1216$ cliques
corresponding to designs, and four designs are non-isomorphic. The
result is stated in the next theorem (nauty~\cite{MP14} was also
used to compute the full automorphism group, and GAP~\cite{GAP4} to
analyse its structure).

\begin{theorem}\label{tm56-16-18}
There are four $(56,16,18)$ QSDs with $G_{48}$ as automorphism
group. Three of them are the designs $D_1$, $D_2$, $D_3$
of~\cite[Theorem 4.1]{KV16}, and the fourth is a new design $D_4$
with full automorphism group of order $192$ isomorphic to
$(((\Z_2)^4 \rtimes \Z_2) \rtimes \Z_2) \rtimes \Z_3$.
\end{theorem}

The block graphs of the four designs are isomorphic to the strongly
regular Cameron graph~\cite{AEB86} with para\-meters $(231,30,9,3)$.
We also classified $(56,16,18)$ QSDs with other subgroups of $H$ of
orders $48$ and $32$, but only found these four designs.

\begin{table}[b]
\begin{tabular}{|c|c|rrrrrrr|}
\cline{2-9} \multicolumn{1}{c|}{} &
\multicolumn{1}{c|}{$\dim$\rule{0mm}{4.5mm}} &
\multicolumn{1}{c}{$a_0$} & \multicolumn{1}{c}{$a_8$} &
\multicolumn{1}{c}{$a_{12}$} & \multicolumn{1}{c}{$a_{16}$} &
\multicolumn{1}{c}{$a_{20}$} &
\multicolumn{1}{c}{$a_{24}$} & \multicolumn{1}{c|}{$a_{28}$}\\
\hline

$C_{1,2}$\rule{0mm}{4.5mm} & 23 & 1 & 75 & 0 & 21\,657 & 353\,536 & 2\,059\,035 & 3\,520\,000 \\

$C_3$ & 19 & 1 & 0 & 0 & 1\,722 & 19\,936 & 134\,085 & 212\,800 \\

$C_4$ & 23 & 1 & 15 & 216 & 20\,493 & 359\,200 & 2\,044\,899 & 3\,538\,960 \\

\hline
\end{tabular}
\vskip 3mm \caption{Dimensions and weight distributions of self-orthogonal binary codes
spanned by $(56,16,18)$ QSDs.}\label{table2}
\end{table}

Let $C_i$ be the binary code spanned by block incidence vectors of
the design $D_i$, for $i=1,2,3,4$. The codes are self-orthogonal,
because the intersection numbers $x=4$, $y=8$ are even. The codes
$C_1$ and $C_2$ are equivalent, of dimension $23$. The code~$C_3$ is
a subcode of dimension $19$. The new code~$C_4$ is also of dimension
$23$, but can be distinguished by the weight enumerator
$W(x)=\sum_{i=0}^{56} a_i x^i$. Coefficients are given in
Table~\ref{table2}, with $a_{56-i}=a_i$ because the codes contain
the all-one vector. The codes were analysed with Magma~\cite{BCP97}.

We tried to construct other $(56,16,18)$ QSDs from words of weight
$16$ using clique search in the associated compatibility graphs, as
in~\cite{MT04}. The $1\,722$ codewords of $C_3$ support only the
design~$D_3$. We could not perform a complete search for the
$21\,657$ codewords of $C_{1,2}$ and the $20\,493$ codewords of
$C_4$. The compatibility graphs have respective
densities $0.5650$ 
and $0.5497$, 
and Cliquer needs much more time than for the previous graph with
$301\,080$ vertices. We did partial searches by identifying
codewords into orbits under various automorphism groups, but only
found the four designs of Theorem~\ref{tm56-16-18}. By using a
dihedral group of order $10$, we got $D_1$, $D_2$ and $D_3$ from the
codewords of $C_{1,2}$. From the codewords of $C_4$ we only got the
design $D_4$, by using automorphism groups of orders as small
as~$3$.

\section{New $(56,16,6)$ QSDs}\label{sec4}

The group $G_{48}$ can also be used to construct new quasi-symmetric
$(56,16,6)$ designs. Among the $867\,693\,085\,859$ orbits of
$16$-element subsets, there are $5352$ good orbits with intersection
numbers $x=4$, $y=6$. The compatibility graph has $5352$ vertices
and $379\,369$ edges (density $0.02649$). Cliquer quickly found
$224\,256$ cliques of weight $b=77$; all of them correspond to QSDs.
Using nauty, we computed the number of non-isomorphic designs and
their full automorphism groups.

\begin{theorem}\label{d56166g48}
There are $876$ quasi-symmetric $(56,16,6)$ designs with $G_{48}$ as
automorphism group. They have $G_{48}$ as their full automorphism
group and are not isomorphic to the two known designs
from~\cite{MT04,VT87}.
\end{theorem}

Next, we want to classify $(56,16,6)$ QSDs with the Frobenius group
$Frob_{21}\cong \Z_7\rtimes \Z_3$ of order $21$. The full automorphism
group of the QSD from~\cite{VT87} is of order $168$ and has a
subgroup isomorphic to $Frob_{21}$, but we want to proceed
systematically and consider all possible actions.

\begin{lemma}
An automorphism of order $7$ of a $(56,16,6)$ QSD does not fix any
points and blocks.
\end{lemma}

\begin{proof}
The automorphism~$\alpha$ maps a block through two fixed points
$F_1$ and $F_2$ on a block through $F_1$ and $F_2$. Since there are
$\lambda=6$ such blocks, and $\alpha$ is of order~$7$, the blocks
through $F_1$ and $F_2$ are fixed by~$\alpha$. Dually, a point on
the intersection of two fixed blocks is fixed, because the
intersection numbers are $x=4$, $y=6$. Any fixed block contains $9$
or $16$ fixed points and through any fixed point there are $8$, $15$
or $22$ fixed blocks. The set of fixed points and blocks of~$\alpha$
is a partially balanced design (PBD) with parameters
$(7m,\{9,16\},6)$ and points of degrees $8$, $15$, and $22$. Let
$v_8$, $v_{15}$, $v_{22}$ be the number of points of the respective
degrees in this PBD, and $b_9$, $b_{16}$ the number of blocks of
degrees $9$ and $16$. Clearly
\begin{equation}\label{lmeq1}
v_8 + v_{15}+v_{22} = 7m,
\end{equation}
and by double counting incident point-line pairs we get
\begin{equation}
8v_8 + 15v_{15}+22v_{22}= 9b_9 + 16b_{16}.
\end{equation}
By double counting triples $(P,Q,B)$ of two points $P$ and $Q$
incident with a block~$B$ we get
\begin{equation}
\textstyle {9\choose 2}b_9+{16\choose 2}b_{16}={7m\choose 2}\cdot 6.
\end{equation}
Finally, through each of the $v_8$ points incident with $8$ fixed
blocks there are $14$ non-fixed blocks, and through each of the
$v_{15}$ points incident with $15$ fixed blocks there are $7$
non-fixed blocks. Together with the fixed blocks, this number cannot
be greater than the total number of blocks:
\begin{equation}\label{lmeq4}
14v_8 + 7v_{15}+b_9+b_{16}\le 77.
\end{equation}
The system of (in)equalities \eqref{lmeq1}-\eqref{lmeq4} is
inconsistent for $m=1,\ldots,7$. Therefore, $m=0$ and $\alpha$ does
not fix any points and blocks.
\end{proof}

Thus, the group $Frob_{21}$ acts in orbits of length $7$ and $21$ on
the points and blocks of a $(56,16,6)$ QSD. The action on each orbit
is unique up to permutational isomorphism. This allows three
possible actions on the $56$ points, with orbit size distributions
$\nu^{(1)}=(7,7,7,7,7,7,7,7)$, $\nu^{(2)}=(7,7,7,7,7,21)$, and
$\nu^{(3)}=(7,7,21,21)$. Our orderly algorithm needs about $7$ CPU
days to generate the orbits of $16$-element subsets:
\begin{itemize}
\item $\nu^{(1)}$: $107\,602\,880$ good orbits ($1\,983\,283\,532\,181$ total orbits),
\item $\nu^{(2)}$: $98\,909\,810$ good orbits ($1\,983\,283\,449\,525$ total orbits),
\item $\nu^{(3)}$: $584\,272\,493$ good orbits ($1\,983\,283\,432\,389$ total orbits).
\end{itemize}

There are far too many good orbits to search for cliques in the
compatibility graphs. We can reduce the number of orbits that need
to be considered without losing generality by using \emph{orbit
matrices}. Let $\mathcal{O}_1,\ldots,\mathcal{O}_m$ be the
point-orbits and $\mathcal{B}_1,\ldots,\mathcal{B}_n$ the
block-orbits of a group $G$ acting on a $(v,k,\lambda)$ design.
Denote the orbit sizes by $\nu_i=|\mathcal{O}_i|$ and
$\beta_j=|\mathcal{B}_j|$; then $\sum_{i=1}^m \nu_i = v$ and
$\sum_{j=1}^n \beta_j = b$. For $G\cong Frob_{21}$ and our $(56,16,6)$
QSD, the possible point-orbit size distributions are $\nu^{(1)}$,
$\nu^{(2)}$, $\nu^{(3)}$ given above, and the block-orbit size
distributions are $\beta^{(1)}=(7,7,7,7,7,7,7,7,7,7,7)$,
$\beta^{(2)}=(7,7,7,7,7,7,7,7,21)$, $\beta^{(3)}=(7,7,7,7,7,21,21)$,
and $\beta^{(4)}=(7,7,21,21,21)$.

Let $a_{ij} = | \{ P\in \mathcal{O}_i \mid P\in B\} |$, for some
$B\in \mathcal{B}_j$. This number does not depend on the choice of
$B$, because the orbits form a tactical decomposition of the design.
The matrix $A=[a_{ij}]$ has the following properties:
\begin{enumerate}
\item[1.] $\sum\limits_{i=1}^m a_{ij} = k$,\\
\item[2.] $\sum\limits_{j=1}^n {\beta_j \over \nu_i} \kern 2pt
a_{ij} = r$,
\item[3.] $\sum\limits_{j=1}^n {\beta_j \over
\nu_{i'}} \kern 2pt a_{ij} a_{i'j} =\left\{\begin{array}{l l}
\lambda \nu_i, & \mbox{for } i \neq i',\\
\lambda(\nu_i-1)+r, & \mbox{for } i=i'.\\
\end{array}\right.$
\end{enumerate}
A matrix with these properties is called an \emph{orbit matrix} for
$(v,k,\lambda)$ and $G$. Orbit matrices were used for the
construction of block designs with prescribed automorphisms in many
papers, e.g.\ \cite{DHLST98, JT85, VK02}.

The entries of an orbit matrix are bounded by $0 \le a_{ij} \le
\nu_i$. In our case we can also exclude entries $a_{ij}=2$ and $5$
whenever $\nu_i=\beta_j=7$. An orbit $\mathcal{B}_j$ of size $7$ is
stabilised by a subgroup of order~$3$ of $Frob_{21}$, which has a
fixed point and two orbits of size $3$ on the $7$ points
of~$\mathcal{O}_i$. Thus, $a_{ij}$ must be a sum of $1$, $3$ and
$3$. Furthermore, for a quasi-symmetric design with intersection
numbers $x$ and $y$, the matrix $A$ has the additional properties
\begin{enumerate}
\item[4.] $\sum\limits_{i=1}^m {\beta_j \over \nu_{i}} \kern 1pt
a_{ij} a_{ij'} = \left\{\begin{array}{ll} sx+(\beta_j-s)y, &
\mbox{for } j \neq j',\,\, 0\le s\le \beta_j,\\
sx+(\beta_j-1-s)y+k, & \mbox{for } j=j',\,\, 0\le s< \beta_j.\\
\end{array}\right.$
\end{enumerate}
An orbit matrix satisfying these equations is called \emph{good}.
In~\cite{DHLST98}, equations 4.\ were used  for the classification
of $(28,12,11)$ QSDs with $x=4$, $y=6$ and an automorphism of
order~$7$.

If the number of columns $n$ is not too large, we can classify all
orbit matrices up to rearrangements of rows and columns by an
orderly Read-Farad\v{z}ev type algorithm described in~\cite{VK02}. We
then check equations 4.\ and the requirement that $a_{ij}\neq 2,5$
whenever $\nu_i=\beta_j=7$. Matrices exist in $4$ of the $12$
combinations of point- and block-orbit size distributions for
$Frob_{21}$ on a $(56,16,6)$ QSD:
\begin{enumerate}
\item $\nu^{(1)}$, $\beta^{(2)}$ $\leadsto$ $2$ orbit matrices,
\item $\nu^{(2)}$, $\beta^{(3)}$ $\leadsto$ $6$ orbit matrices,
\item $\nu^{(3)}$, $\beta^{(3)}$ $\leadsto$ $4$ orbit matrices,
\item $\nu^{(3)}$, $\beta^{(4)}$ $\leadsto$ $1$ orbit matrix.
\end{enumerate}

In case (4), the orbit matrix is
$$A=\left[\begin{array}{ccccc}
4 & 0 & 3 & 2 & 1 \\
0 & 4 & 3 & 2 & 1 \\
9 & 9 & 4 & 6 & 6 \\
3 & 3 & 6 & 6 & 8\\
\end{array}\right].$$
The $j$-th column of this matrix tells us how the points on a block
of~$\mathcal{B}_j$ are distributed among the point-orbits
$\mathcal{O}_1,\ldots,\mathcal{O}_m$. We can adapt the algorithm
from Section~\ref{sec3} to search for orbits compatible with a
column of~$A$: simply add the conditions $|(U\cup\{e\})\cap
\mathcal{O}_i|\le a_{ij}$, $i=1,\ldots,m$ to the \textbf{if}
statement in line~8. In our case, good orbits with the required
intersection pattern exist only for the fourth column of~$A$. For
the other columns, there are no compatible good orbits and therefore
QSDs corresponding to this orbit matrix do not exist.

Similarly, for the $8$ orbit matrices of cases (1) and (2),
compatible good orbits do not exist for at least one column. Only
the orbit matrices of case (3) allow good orbits for every column.
Here are two of the four matrices, transposed, with numbers of
compatible good orbits:
$$A^{\tau}=\left[\begin{array}{cccc}
 4 & 0 & 6 & 6 \\
 4 & 3 & 6 & 3 \\
 3 & 1 & 3 & 9 \\
 1 & 3 & 3 & 9 \\
 1 & 0 & 9 & 6 \\
 2 & 3 & 6 & 5 \\
 1 & 2 & 7 & 6 \\
\end{array}\right],\kern 3mm
\left[\begin{array}{cccc}
 4 & 0 & 6 & 6 \\
 3 & 4 & 6 & 3 \\
 3 & 1 & 9 & 3 \\
 3 & 1 & 3 & 9 \\
 0 & 1 & 9 & 6 \\
 2 & 3 & 5 & 6 \\
 1 & 2 & 6 & 7 \\
\end{array}\right]\kern 3mm \begin{array}{lr}
\to & 882 \mbox{ orbits,}\\
\to & 588 \mbox{ orbits,}\\
\to & 490 \mbox{ orbits,}\\
\to & 490 \mbox{ orbits,}\\
\to & 735 \mbox{ orbits,}\\
\to & 3\,674\,412 \mbox{ orbits,}\\
\to & 3\,628\,548 \mbox{ orbits.}\\
\end{array}$$
We still get millions of good orbits of size $\beta_6=\beta_7=21$
from the last two columns, but now we have information on how they
must be chosen. The design is comprised of one orbit from each set
compatible with a column of the orbit matrix. A backtracking program
written in C can complete the search. The left matrix gives rise to
the known $(56,16,6)$ QSD, and the right matrix gives a new design.
The remaining two orbit matrices of case (3) do not yield designs.
Thus, we can conclude

\begin{theorem}\label{tm16-6frob21}
There are two $(56,16,6)$ QSDs with an automorphism group isomorphic
to $Frob_{21}$. One is the known design of~\cite{VT87} with full
automorphism group of order~$168$, and the other one is a new design
with full automorphism group of order~$21$.
\end{theorem}

The computational proof of Theorem~\ref{tm16-6frob21} is summarised
by the following flowchart.

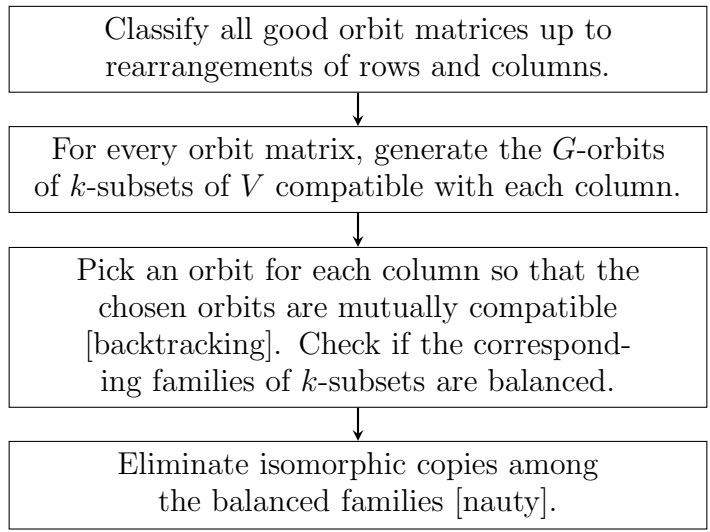
\begin{center}
\begin{tikzpicture}[node distance=1.8cm]
\node (step1) [process] {Classify all good orbit matrices up to
rearrangements of rows and columns.};
\node (step2) [process, below of=step1, yshift=2mm] {For every orbit matrix,
generate the $G$-orbits of $k$-subsets of $V$ compatible with each column.};
\node (step3) [process, below of=step2, yshift=-3mm] {Pick an orbit for each
column so that the chosen orbits are mutually compatible [backtracking]. Check
if the corresponding families of $k$-subsets are balanced.};
\node (step4) [process, below of=step3, yshift=-3mm] {Eliminate isomorphic
copies among the balanced families [nauty].};
\draw [arrow] (step1) -- (step2);
\draw [arrow] (step2) -- (step3);
\draw [arrow] (step3) -- (step4);
\end{tikzpicture}
\end{center}

We tried to construct more $(56,16,6)$ QSDs with the alternating
group $A_4$ of order $12$. The two designs of~\cite{MT04,VT87} and
the $876$ designs of Theorem~\ref{d56166g48} allow two actions
of~$A_4$, with orbit size distributions $\nu^{(1)}=(4, 4, 6, 6, 6,
6, 12, 12)$, $\beta^{(1)}=(1, 1, 1, 3, 3, 4, 4, 6, 6, 12, 12, 12,
12)$ and $\nu^{(2)}=(1, 3, 4, 6, 6, 12, 12, 12)$, $\beta^{(2)}=(1,
1, 3, 4,4, 4, 6, 6, 6, 6, 12,$ $12, 12)$. There are $3$ good orbit
matrices in the first case and $13$ in the second case. Now our
backtracking search for designs from orbits compatible with an orbit
matrix takes considerably more time, and we could not complete the
search. However, we did find new designs with $A_4$ as their full
automorphism group: $67$ in the first case and $236$ in the second
case. Since we performed an incomplete search, there are probably
more designs in both cases.

\begin{table}[b]
{\small
\begin{tabular}{|c|c|rrrrrrr|}
\cline{2-9} \multicolumn{1}{c|}{} &
\multicolumn{1}{c|}{$\dim$\rule{0mm}{4.5mm}} &
\multicolumn{1}{c}{$a_0$} & \multicolumn{1}{c}{$a_8$} &
\multicolumn{1}{c}{$a_{12}$} & \multicolumn{1}{c}{$a_{16}$} &
\multicolumn{1}{c}{$a_{20}$} &
\multicolumn{1}{c}{$a_{24}$} & \multicolumn{1}{c|}{$a_{28}$}\\
\hline

$C_1$\rule{0mm}{4.5mm} & 26 & 1 & 91 & 2\,016 & 152\,425 &
2\,939\,776 & 16\,194\,619 & 28\,531\,008 \\

$C_2$ & 26 & 1 & 7 & $2\,016$ & $155\,365$ & $2\,926\,336$ & $16\,224\,019$ & $28\,493\,376$ \\

$C_3$ & 24 & 1 & 75 & 0 & 40\,089 & 730\,368 & 4\,055\,835 & 7\,124\,480 \\

$C_{4\hbox{-}6,9,10}$ & 22 & 1 & 15 & 0 & 9\,933 & 183\,168 & 1\,012\,515 & 1\,783\,040 \\

$C_{7,11\hbox{-}13}$ & 25 & 1 & 75 & 672 & 77\,721 & 1\,465\,984 & 8\,103\,963 & 14\,257\,600 \\

$C_8$ & 25 & 1 & 75 & 960 & 75\,417 & 1\,474\,048 & 8\,087\,835 & 14\,277\,760 \\

$C_{14}$ & 22 & 1 & 15 & 0 & 10\,701 & 178\,560 & 1\,024\,035 & 1\,767\,680 \\

$C_{15}$ & 23 & 1 & 15 & 288 & 19\,917 & 361\,216 & 2\,040\,867 & 3\,544\,000 \\

$C_{16}$ & 23 & 1 & 15 & 96 & 19\,917 & 365\,056 & 2\,028\,579 & 3\,561\,280 \\

$C_{17}$ & 24 & 1 & 75 & 160 & 39\,833 & 728\,704 & 4\,062\,235 & 7\,115\,200 \\

$C_{18}$ & 22 & 1 & 15 & 64 & 9\,677 & 183\,424 & 1\,012\,771 & 1\,782\,400 \\

$C_{19,21,24}$ & 22 & 1 & 15 & 16 & 10\,061 & 182\,080 & 1\,015\,459 & 1\,779\,040 \\

$C_{20,22}$ & 22 & 1 & 15 & 64 & 10\,445 & 178\,816 & 1\,024\,291 & 1\,767\,040 \\

$C_{23}$ & 25 & 1 & 75 & 1\,280 & 74\,905 & 1\,470\,720 & 8\,100\,635 & 14\,259\,200 \\

$C_{25}$ & 25 & 1 & 75 & 992 & 77\,209 & 1\,462\,656 & 8\,116\,763 & 14\,239\,040 \\

$C_{26}$ & 27 & 1 & 139 & 4\,992 & 307\,161 & 5\,848\,832 & 32\,477\,083 & 56\,941\,312 \\

$C_{27}$ & 27 & 1 & 99 & 4\,304 & 305\,873 & 5\,872\,320 & 32\,406\,731 & 57\,039\,072 \\

$C_{28,29}$ & 27 & 1 & 99 & 4\,112 & 307\,409 & 5\,866\,944 & 32\,417\,483 & 57\,025\,632 \\

$C_{30}$ & 26 & 1 & 147 & 1\,008 & 158\,529 & 2\,920\,512 & 16\,231\,467 & 28\,485\,536 \\

$C_{31,32,34,35}$ & 27 & 1 & 147 & 3\,696 & 309\,057 & 5\,862\,976 & 32\,423\,979 & 57\,018\,016 \\

$C_{33,39}$ & 27 & 1 & 147 & 4\,976 & 307\,009 & 5\,849\,664 & 32\,475\,179 & 56\,943\,776 \\

$C_{36}$ & 26 & 1 & 75 & 2\,240 & 153\,241 & 2\,931\,200 & 16\,218\,395 & 28\,498\,560 \\

$C_{37,38}$ & 27 & 1 & 75 & 4\,416 & 305\,817 & 5\,871\,616 & 32\,408\,859 & 57\,036\,160 \\

\hline
\end{tabular} }
\vskip 3mm \caption{Dimensions and weight distributions of self-orthogonal binary codes
spanned by $(56,16,6)$ QSDs.}\label{table3}
\end{table}

We can further increase the number of known $(56,16,6)$ QSDs by
considering the associated self-orthogonal binary codes, in the
spirit of~\cite{MT04}. The designs of~\cite{MT04,VT87} span a code
$C_1$ of dimension $26$. The design of Theorem~\ref{tm16-6frob21}
with full automorphism group $Frob_{21}$ spans an inequivalent code
$C_2$, also of dimension $26$. The $876$ designs of
Theorem~\ref{d56166g48} span $23$ inequivalent codes
$C_3,\ldots,C_{25}$ of dimensions $22$--$25$. Finally, the $303$
designs with full automorphism group $A_4$ span $16$ inequivalent
codes. Two of them are equivalent to the previous codes $C_3$ and
$C_8$, and the others are new codes $C_{26},\ldots,C_{39}$ of
dimensions $26$ and $27$. Weight enumerators of the codes are given
in Table~\ref{table3}. Some inequivalent codes have equal weight
distributions; in total $23$ different weight enumerators occur. The
computation was done in Magma.

We managed to find $228$ more $(56,16,6)$ QSDs with full
automorphism groups of order $16$ by searching among words of weight
$16$ of the codes $C_3,\ldots,C_{25}$ with Cliquer. By using nauty
once more, we can conclude

\begin{theorem}\label{new56166}
There are at least $1410$ quasi-symmetric $(56,16,6)$ designs. Their
distribution by order of full automorphism group is given in
Table~\ref{table4}.
\end{theorem}

\begin{table}[h]
\begin{tabular}{c | c c}
$|\Aut|$ & \#$(56,16,6)$ & \#$(78,22,6)$ \\
\hline
$168$\rule{0mm}{4.5mm} & $1$ & $2$\\
$78$ & $0$ & $1$\\
$48$ & $876$ &  $1664$\\
$24$ & $1$ & $378$\\
$21$ & $1$ & $2$\\
$16$ & $228$ & $456$\\
$12$ & $303$ & $606$\\
$6$ & $0$ & $32$\\
\hline
\end{tabular}
\vskip 3mm \caption{Distribution of the known $(56,16,6)$ QSDs and
symmetric $(78,22,6)$ designs by full automorphism group
order.}\label{table4}
\end{table}

The block graphs of the QSDs are strongly regular with parameters
$(77,16,0,4)$. Such a graph is unique~\cite{AEB83} and this means
that all the QSDs can be embedded as residuals of symmetric
$(78,22,6)$ designs, as noted by Tonchev~\cite{VT87}. The first
example of symmetric $(78,22,6)$ designs was constructed
in~\cite{JT85} and it is self-dual. Two more dual pairs were
obtained by embedding QSDs in~\cite{MT04, VT87}. More examples were
constructed in~\cite{CDR18} by assuming an automorphism of
order~$6$, bringing the number of known $(78,22,6)$ designs up to
$413$. This number is now further increased by embedding the new
QSDs from Theorem~\ref{new56166}.

\begin{theorem}
There are at least $3141$ symmetric $(78,22,6)$ designs. Their
distribution by order of full automorphism group is given in
Table~\ref{table4}.
\end{theorem}

The design from~\cite{JT85} is still the only known self-dual
$(78,22,6)$ design, and the other examples form $1570$ dual pairs.

\section{Nonexistence of $(56,12,9)$ and $(56,20,19)$ QSDs
with certain automorphism groups}\label{sec5}

In this section, we consider designs with parameters from rows 49
and 51 of~\cite[Table 48.25]{MS07} and groups that yielded new
designs in the previous sections. Acting on $12$-element subsets,
the group $G_{48}$ has only $12$ good orbits with intersection
numbers $x=0$, $y=3$, of sizes $1$ and $3$. A $(56,12,9)$ QSD would
have $b=210$ blocks, and thus clearly does not allow $G_{48}$ as
automorphism group. The same holds for subgroups of $G_{48}$ of
orders $16$ and $8$. They have only small numbers of good orbits
($180$ and $740$, respectively), of sizes $1$ and $2$, which cannot
be used to build $(56,12,9)$ QSDs.

Subgroups of order $12$ are isomorphic to $A_4$ and have
significantly more good orbits, including ``long orbits'' of
size~$12$. Two actions on the $56$ points occur, with point-orbit
size distributions $\nu^{(1)}$ and $\nu^{(2)}$ given earlier.
There are $3\,148\,236$ good orbits of $12$-element subsets
for~$\nu^{(1)}$, among them $16\,588$ ``short orbits'' of size less
than $12$, and $5\,664\,770$ for $\nu^{(2)}$, among them $53\,954$
short orbits. The compatibility graphs are to large to invoke
Cliquer directly, but we can proceed in the following way. Since
$b=210$ is not divisible by $12$, the design must have at least one
short block-orbit. For every short orbit $\mathcal{B}$ of
$12$-element subsets, we consider all orbits compatible with
$\mathcal{B}$ and search for cliques in the corresponding subgraph
of the compatibility graph. Cliquer could eliminate cliques of
weight $b-|\mathcal{B}|$ in all ensuing subgraphs in a few days of
CPU time. Since $G_{48}$ has no subgroups of order $24$, we can
conclude

\begin{theorem}\label{tm56-12-9g48}
Let $G$ be a subgroup of order at least $8$ of the permutation group
$G_{48}$. Then, quasi-symmetric $(56,12,9)$ designs with $G$ as
automorphism group do not exist.
\end{theorem}

Next, we consider $Frob_{21}$ as an automorphism group of
$(56,12,9)$ QSDs. Again, we want to consider all possible actions on
the $56$ points. To reduce the number of possibilities, we need the
following

\begin{lemma}
An automorphism of order $7$ of a $(56,12,9)$ QSD does not fix any
points and blocks.
\end{lemma}

\begin{proof}
Let $\alpha$ be an automorphism of order $7$ fixing $7m$ points. A
fixed block contains $5$ or $12$ fixed points; denote the number of
such blocks by $b_5$ and $b_{12}$. If $m=1$, then $b_{12}=0$, $b_5$
is at least $7$, and any pair of fixed blocks intersects in $y=3$
fixed points. This would yield a family of $5$-subsets of the $7$
fixed points pairwise intersecting in $3$ points. Such a family can
have at most $3$ subsets, and thus $m\ge 2$. The number of
point-orbits of size~$7$ is $8-m$, and since every $b_5$-block
contains such an orbit, we have $b_5\le 6$.

Consider the $\lambda=9$ blocks through two fixed points $F_1$ and
$F_2$. Either two of them are fixed and $7$ are in a block-orbit
$\mathcal{B}$, or all $9$ are fixed blocks. A pair of blocks from
$\mathcal{B}$ intersects in the fixed points $F_1$ and $F_2$, and
therefore must have a third intersection point $T$. If $T$ is fixed,
the remaining $k-3=9$ points on a block $B\in \mathcal{B}$ belong to
different point-orbits of size $7$. This is not possible because we
have at most $6$ point-orbits of size $7$. If $T$ is not fixed, a
block $B\in \mathcal{B}$ contains at most $3$ points from each orbit
$\mathcal{O}$ of size~$7$. If $B$ contains $3$ points from one orbit
$\mathcal{O}$, then $\mathcal{O}$ and $\mathcal{B}$ form a Fano
plane and the remaining $k-5=7$ points on $B$ belong to different
point-orbits of size~$7$. The only other possibility is for $B$ to
contain two points from $3$ orbits of size~$7$, and one point from
further $k-8=4$ orbits of size~$7$. Both possibilities would require
more than $6$ point-orbits of size~$7$.

Hence, there are $9$ fixed blocks through $F_1$ and $F_2$ and the
set of fixed points and blocks forms a PBD with parameters
$(7m,\{5,12\},9)$. Double counting triples $(P,Q,B)$ of two fixed
points incident with a fixed block yields
\begin{equation}\label{eq56129}
\textstyle {5\choose 2}b_5+{12\choose 2}b_{12}={7m\choose 2}\cdot 9.
\end{equation}
From \eqref{eq56129}, $b_{12}=\frac{63 m (7 m-1)-20b_5}{132}$ and
this expression is not an integer for $2\le m\le 7$ and $0\le b_5\le
6$. Therefore, $m=0$ and $\alpha$ has no fixed points and blocks.
\end{proof}

We need to consider three possible actions of $Frob_{21}$ on the
$56$ points, with orbit size distributions $\nu^{(1)}$, $\nu^{(2)}$,
$\nu^{(3)}$ as in Section~\ref{sec4}. We computed the orbits of
$12$-element subsets with the algorithm from Section~\ref{sec3}:
\begin{itemize}
\item $\nu^{(1)}$: $5\,824$ good orbits ($26\,589\,705\,660$ total orbits),
\item $\nu^{(2)}$: $459\,550$ good orbits ($26\,589\,687\,420$ total orbits),
\item $\nu^{(3)}$: $53\,578$ good orbits ($26\,589\,683\,340$ total orbits).
\end{itemize}
The compatibility graphs have densities $2.173\cdot 10^{-2}$,
$2.328\cdot 10^{-5}$, and $3.612\cdot 10^{-5}$, respectively. We
used Cliquer to establish that the maximum weight of a clique in the
graphs are $21$, $77$, and $35$. This is less than the required
number of blocks $b=210$.

\begin{theorem}
Quasi-symmetric $(56,12,9)$ designs with an automorphism group
isomorphic to $Frob_{21}$ do not exist.
\end{theorem}

Finally, we turn to $(56,20,19)$ QSDs with intersection numbers
$x=5$, $y=8$ and $G_{48}$ as automorphism group. Generating orbits
of $20$-element subsets with the algorithm from Section~\ref{sec3}
would take about ${56\choose 20}/{56\choose 16}\approx 19$ times
longer than for the designs with $k=16$, and only a tiny fraction of
the orbits are expected to be good. We can speed up the computation
by adding the condition $|U\cap g(U)|\le y$, for $g\in G\setminus
\{1\}$, to the \textbf{if} statement in line~8. A set $U$
intersecting its image $g(U)$ in more than $y$ points cannot be
extended to a $k$-element representative of a good orbit, unless the
extension $U'\supset U$ belongs to a short orbit. The new condition
causes the algorithm to miss the short good orbits, but we can
generate them quickly by the algorithm from~\cite{KV16}. It took
about $2$ CPU days to generate the $384$ long good orbits (of size
$48$) and a few more minutes for the $3851$ short good orbits (of
size less than $48$). The compatibility graph has $4235$ vertices
and $163\,766$ edges (density $0.01827$). The maximum weight of a
clique is $142$,  less than the required number of blocks $b=154$.

\begin{theorem}\label{tm56-20}
Quasi-symmetric $(56,20,19)$ designs with $G_{48}$ as automorphism
group do not exist.
\end{theorem}

Assuming a smaller automorphism group, e.g.\ a subgroup of $G_{48}$
or some permutation representation of $Frob_{21}$ on $56$ points,
increases the number of good orbits quite dramatically. We can no
longer perform the second step of the computation, neither by clique
search nor by using orbit matrices.

\section{Conclusion}\label{sec6}

New bounds on numbers of non-isomorphic quasi-symmetric designs on
$56$ points are given in Table~\ref{table5}. The incidence matrices
of the constructed designs can be downloaded from our web page
\begin{center}
{\small
\url{https://web.math.pmf.unizg.hr/~krcko/results/quasisym.html} }
\end{center}
Previously, only two $(56,16,6)$ QSDs~\cite{MT04} and three
$(56,16,18)$ QSDs~\cite{KV16} were known. Almost any approach we
tried for $(56,16,6)$ QSDs increased the number of known designs. On
the other hand, only one new $(56,16,18)$ QSD was found, and the
existence of $(56,12,9)$ and $(56,20,19)$ QSDs remains open.

\begin{table}[!h]\label{table5}
\begin{tabular}{ccccccccc}
\hline
No. & $v$ & $k$ & $\lambda$ & $r$ & $b$ & $x$ & $y$ & NQSD\\
\hline
47 & 56 & 16 & 18 & 66 & 231 & 4 & 8 & $\ge 4$\\
48 & 56 & 15 & 42 & 165 & 616 & 3 & 6 & 0\\
49 & 56 & 12 & 9 & 45 & 210 & 0 & 3 & ?\\
50 & 56 & 21 & 24 & 66 & 176 & 6 & 9 & 0\\
51 & 56 & 20 & 19 & 55 & 154 & 5 & 8 & ?\\
52 & 56 & 16 & 6 & 22 & 77 & 4 & 6 & $\ge 1410$\\
\hline
\end{tabular}
\vskip 3mm \caption{An updated table of QSDs with
$v=56$.}\label{table5}
\end{table}

Regarding computational techniques, the algorithm from
Section~\ref{sec3} solves the problem of generating orbit
representatives of $k$-element subsets satisfactorily. It can be
adapted to generate orbits compatible with an orbit matrix (see
Section~\ref{sec4}), and speed-up is possible when the intersection
number~$y$ is comparatively small (see the proof of
Theorem~\ref{tm56-20}). The main computational problem remains
putting the orbits or individual $k$-subsets together to form QSDs.

The requirement that every pair of $k$-subsets intersects in~$x$
or~$y$ points proved stronger than the requirement that they cover
every $2$-subset exactly $\lambda$ times. We could simply ignore the
second requirement until the end of the computation. In some cases
all the constructed structures were balanced
(Theorem~\ref{d56166g48}), and in others the number of non-balanced
structures was not too large (Theorem~\ref{tm56-16-18}). The
opposite approach would give many more block designs that are not
quasi-symmetric.

The critical factor for the computation is the number of candidates
for blocks of the design (good orbits, or individual $k$-subsets
obtained from codes). The difficulty of the problem also depends on
the number of compatible candidates, i.e.\ candidates intersecting
in $x$ or $y$ points, measured by the density of the compatibility
graph. For low densities we could handle problems with hundreds of
thousands or even millions of candidates
(Theorem~\ref{tm56-12-9g48}), and for higher densities problems with
only a few thousand candidates. The number of candidates can
sometimes be reduced by considering orbit matrices, as in the proof
of Theorem~\ref{tm16-6frob21}. For this approach the number of
blocks of the design must be small enough to allow complete
classification of the orbit matrices.

Despite our efforts in Section~\ref{sec5}, we did not find any
$(56,12,9)$ and $(56,20,19)$ QSDs. If these designs exist, we
believe some new insight or completely different computational
approach will be necessary for their construction.


\begin{thebibliography}{20}

\bibitem{BJL99}
T.~Beth, D.~Jungnickel, H.~Lenz, \emph{Design theory, second
edition}, Cambridge University Press, 1999.

\bibitem{BCP97}
W.~Bosma, J.~Cannon, C.~Playoust, \emph{The Magma algebra system. I.
The user language}, J.\ Symbolic Comput.\ \textbf{24} (1997), no.\
3--4, 235--265.

\bibitem{AEB83}
A.~E.~Brouwer, \emph{The uniqueness of the strongly regular graph on
$77$ points}, J. Graph Theory \textbf{7} (1983), 455--461.

\bibitem{AEB86}
A.~E.~Brouwer, \emph{Uniqueness and nonexistence of some graphs
related to $M_{22}$}, Graphs Combin.\ \textbf{2} (1986), no.\ 1,
21–-29.

\bibitem{AC88}
A.~R.~Calderbank, \emph{Geometric invariants for quasisymmetric
designs}, J.\ Combin.\ Theory Ser.\ A \textbf{47} (1988), no.\ 1,
101--110.

\bibitem{CDR18}
D.~Crnkovi\'{c}, D.~Dumi\v{c}i\'{c} Danilovi\'{c}, S.~Rukavina, \emph{On symmetric
$(78,22,6)$ designs and related self-orthogonal codes}, Util.\
Math.\ \textbf{109} (2018), 227--253.

\bibitem{CRRT17}
D.~Crnkovi\'{c}, B.~G.~Rodrigues, S.~Rukavina, V.~D.~Tonchev,
\emph{Quasi-symmetric $2$-$(64,24,46)$ designs derived from
$AG(3,4)$}, Discrete Math.\ \textbf{340} (2017), no.\ 10,
2472--2478.

\bibitem{DHLST98}
Y.~Ding, S.~Houghten, C.~Lam, S.~Smith, L.~Thiel, V.~D.~Tonchev,
\emph{Quasi-symmetric $2$-$(28,12,11)$ designs with an automorphism
of order $7$}, J.\ Combin.\ Des.\ \textbf{6} (1998), no.\ 3,
213-–223.

\bibitem{IF78}
I.~A.~Farad\v{z}ev, \emph{Constructive enumeration of combinatorial
objects}. Problemes combinatoires et th\'{e}orie des graphes
(Colloq.\ Internat.\ CNRS, Univ.\ Orsay, Orsay, 1976), pp.\
131--135, Colloq.\ Internat.\ CNRS, 260, CNRS, Paris, 1978.

\bibitem{GAP4}
The GAP~Group, \emph{GAP -- Groups, Algorithms, and Programming,
Version 4.8.10}, 2018, \verb|http://www.gap-system.org|.

\bibitem{JT85} Z.\ Janko, T.\ van Trung, \emph{Construction of a new
symmetric block design for $(78,22,6)$ with the help of tactical
decompositions}, J.\ Combin.\ Theory Ser.~A  {\bf 40} (1985),
451--455.

\bibitem{VK02} V.~Kr\v{c}adinac, \emph{Steiner $2$-designs $S(2,4,28)$ with
nontrivial automorphisms}, Glas.\ Mat.\ Ser.\ III {\bf 37(57)}
(2002), 259--268.

\bibitem{KV16}
V.~Kr\v{c}adinac, R.~Vlahovi\'{c}, \emph{New quasi-symmetric designs by the
Kramer-Mesner method}, Discrete Math.\ \textbf{339} (2016), no.\ 12,
2884--2890.

\bibitem{BM98}
B.~D.~McKay, \emph{Isomorph-free exhaustive generation}, J.\
Algorithms \textbf{26} (1998), no.\ 2, 306--324.

\bibitem{MP14}
B.~D.~McKay, A.~Piperno, \emph{Practical graph isomorphism, II}, J.\
Symbolic Comput.\ \textbf{60} (2014), 94–-112.

\bibitem{MT04}
A.~Munemasa, V.~D.~Tonchev, \emph{A new quasi-symmetric
$2$-$(56,16,6)$ design obtained from codes}, Discrete Math.\
\textbf{284} (2004), no.\ 1-3, 231--234.

\bibitem{AN82}
A.~Neumaier, \emph{Regular sets and quasisymmetric 2-designs}, in:
\emph{Combinatorial theory (Schloss Rauischholzhausen, 1982)},
Lecture Notes in Math., \textbf{969}, Springer, 1982, pp.\ 258-–275.

\bibitem{NO03}
S.~Niskanen, P.~R.~J.~\"{O}sterg\aa{}rd, \emph{Cliquer user's guide,
cersion 1.0}, Communications Laboratory, Helsinki University of
Technology, Espoo, Finland, Tech.\ Rep.\ T48, 2003.

\bibitem{PO02}
P.~R.~J.~\"{O}sterg\aa{}rd, \emph{A fast algorithm for the maximum
clique problem}, Discrete Appl.\ Math.\ \textbf{120} (2002), no.\
1--3, 197--207.

\bibitem{PSN15}
R.~M.~Pawale, M.~S.~Shrikhande, S.~M.~Nyayate, \emph{Conditions for
the parameters of the block graph of quasi-symmetric designs},
Electron.\ J.\ Combin.\ \textbf{22} (2015), no.\ 1, Paper 1.36, 30
pp.

\bibitem{RR78}
R.~C.~Read, \emph{Every one a winner or how to avoid isomorphism
search when cataloguing combinatorial configurations}. Algorithmic
aspects of combinatorics (Conf., Vancouver Island, B.C., 1976).
Ann.\ Discrete Math.\ \textbf{2} (1978), 107--120.

\bibitem{MS07}
M.~S.~Shrikhande, \emph{Quasi-symmetric designs}, in: \emph{The
Handbook of Combinatorial Designs, Second Edition} (eds.\
C.~J.~Colbourn and J.~H.~Dinitz), CRC Press, 2007, pp.\ 578--582.

\bibitem{SS91}
M.~S.~Shrikhande, S.~S.~Sane, \emph{Quasi-symmetric designs},
Cambridge University Press, 1991.

\bibitem{VT87}
V.~D.~Tonchev, \emph{Embedding of the Witt-Mathieu system
$S(3,6,22)$ in a symmetric $2$-$(78,22,6)$ design}, Geom.\ Dedicata
\textbf{22} (1987), no.\ 1, 49--75.

\end{thebibliography}
\end{document}